\documentclass[a4paper, 11pt]{article}

\usepackage[margin=1.2in]{geometry}
\usepackage{parskip}
\usepackage[titletoc,title]{appendix}

\usepackage{array}
\usepackage{bm}
\usepackage{caption}
\usepackage{enumerate}
\usepackage{rotating}
\usepackage[hidelinks, colorlinks = true, linkcolor = blue, urlcolor = blue, anchorcolor = blue, citecolor = blue]{hyperref}
\usepackage{blkarray, bigstrut}
\usepackage{multirow}
\usepackage{booktabs}
\usepackage{siunitx}
\usepackage{mathtools}
\usepackage{amsmath}
\usepackage{amssymb}
\usepackage{amsthm}
\usepackage{complexity}
\usepackage{natbib}
\usepackage{graphicx}
\usepackage{pgf}
\usepackage{tikz}
\usetikzlibrary{matrix,patterns,positioning,decorations,arrows,shapes,decorations.markings,calc,fit}
\usepackage[latin1]{inputenc}
\usepackage[T1]{fontenc}
\usepackage{verbatim}
\usepackage{subfiles}
\usepackage{pgfplots}
\usepackage{color}
\usepackage{float}
\usepackage{enumerate}
\usepackage{emptypage}
\usepackage{subcaption}
\usepackage{setspace}
\usepackage[export]{adjustbox}
\usepackage{subfiles}
\usepackage[algoruled]{algorithm2e}
\usepackage{tabularx, environ}
\usepackage{todonotes}
\setuptodonotes{inline}

\newtheorem{thm}{Theorem}
\newtheorem{defin}[thm]{Definition}

\newtheorem{prop}[thm]{Proposition}
\newtheorem{lem}[thm]{Lemma}

\newtheorem{ex}[thm]{Example}

\newcommand{\mbb}{\mathbb}
\newcommand{\mc}{\mathcal}

\usepackage[normalem]{ulem}

\makeatletter
\newcommand{\problemtitle}[1]{\gdef\@problemtitle{#1}}
\newcommand{\probleminput}[1]{\gdef\@probleminput{#1}}
\newcommand{\problemquestion}[1]{\gdef\@problemquestion{#1}}
\NewEnviron{problem}{
  \problemtitle{}\probleminput{}\problemquestion{}
  \BODY
  \par\addvspace{.5\baselineskip}
  \noindent
  \begin{tabularx}{\textwidth}{@{\hspace{\parindent}} l X c}
    \multicolumn{2}{@{\hspace{\parindent}}l}{\@problemtitle} \\
    \textbf{Input:} & \@probleminput \\
    \textbf{Question:} & \@problemquestion
  \end{tabularx}
  \par\addvspace{.5\baselineskip}
}
\makeatother

\title{Pareto-Optimal Linear Programming}

\author{Bart van Rossum$^{1,\star}$, Twan Dollevoet$^{2}$
\vspace{0.1cm}\\
\small{$^1$Operations, Planning, Accounting \& Control}\\ 
\small{Eindhoven University of Technology, The Netherlands}
\vspace{0.1cm}\\
\small{$^2$Econometric Institute}\\ 
\small{Erasmus University Rotterdam, The Netherlands}
\vspace{0.1cm}\\
\small{$^\star$Corresponding author} \\
\vspace{0.1cm} \\
\small{b.t.c.v.rossum@tue.nl, dollevoet@ese.eur.nl}
}

\date{}

\begin{document}

\tikzset{->-/.style={decoration={
  markings,
  mark=at position #1 with {\arrow{>}}},postaction={decorate}}}

\tikzstyle{block} = [rectangle, draw, 
 text centered, rounded corners, inner sep = 0.2cm]

\tikzset{->-/.style={decoration={
  markings,
  mark=at position #1 with {\arrow{>}}},postaction={decorate}}}
  
\maketitle

\begin{abstract}
\noindent 
Pareto-optimality plays a central role in evaluating the efficiency of solutions to allocation problems, such as house allocation, school choice, and kidney exchange. We introduce a general linear programming problem subject to Pareto-optimality conditions, which we call \textsc{Max-Pareto}. Using the novel result that Pareto-optimal bipartite matchings are fractionally Pareto-optimal, we prove that \textsc{Max-Pareto} is $\mathcal{NP}$-complete. We propose a bilinear programming formulation of \textsc{Max-Pareto}, and evaluate its computational performance on the problem of finding Pareto-optimal allocations of highest welfare. 
\vspace{5mm}
\end{abstract}

\section{Introduction}
\label{sec:intro}

The concept of Pareto-optimality describes societal outcomes in which no person can be made better off without making at least one person worse off. This well-known efficiency measure is frequently studied in the context of allocation problems, where a set of objects must be assigned to agents. Important applications include house allocation \citep{abraham2004pareto}, school allocation \citep{abdulkadirouglu2003school}, and kidney exchange \citep{sonmez2014altruistically}. Since the solutions to allocation problems can be represented as matchings in (bipartite) graphs, it features a close connection with matching theory.

Next to Pareto-optimality, it is common to consider the social welfare of an allocation. This can represent a valuation of the object-agent assignments, reflecting, for example, compatibilities or transportation costs. The social welfare objective typically conflicts with that of Pareto-optimality: an allocation that maximises social welfare is not necessarily Pareto-optimal, and a Pareto-optimal allocation does not necessarily maximise social welfare. As such, it makes sense to maximise welfare among Pareto-optimal allocations. 

\citet{biro2021complexity} show that it is $\mc{NP}$-hard to determine a Pareto-optimal allocation, among all Pareto-optimal allocations, that maximises social welfare. This holds even under object-based weights and complete preferences. \citet{aziz2023computing} consider the problem of maximising welfare among `fair' allocations, and show that this problem is $\mc{NP}$-hard for several fairness measures, including envy-freeness and proportionality. Welfare maximisation has also proven to be $\mc{NP}$-hard when one considers stable fractional matchings \citep{caragiannis2019stable}. Omitting the welfare objective, computing a Pareto-optimal and almost proportional allocation can be done in strongly polynomial time \citep{aziz2020polynomial}.

In this work, we move beyond allocations and consider Pareto-optimal solutions to general linear programming problems. In particular, we assume that each feasible solution maps to a payoff vector, representing the payoffs that a finite number of agents derive from the solution. We are interested in maximising a linear function, e.g., some social welfare function, over the set of Pareto-optimal payoff vectors. We refer to this problem as \textsc{Max-Pareto}. The linear programming model provides a flexible framework to model a wide class of problems. As we will see, this framework can also be used to model discrete allocation problems.

Intuitively, one might suspect that relaxing integrality makes it easier to determine welfare-maximising Pareto-optimal solutions. Perhaps surprisingly, however, our main contribution is to show that the decision version of \textsc{Max-Pareto} remains $\mc{NP}$-complete. To obtain this result, we show that Pareto-optimal matchings on bipartite graphs always satisfy the stronger notion of fractional Pareto-optimality. Informally, this result enables us to model the allocation problems discussed above, including that of \citep{biro2021complexity}, as instances of \textsc{Max-Pareto}. The hardness result directly follows. Finally, we propose a bilinear programming formulation of \textsc{Max-Pareto}, and evaluate its computational performance on the problem of finding Pareto-optimal allocations of highest welfare. On some instances featuring a large number of objects, our generic method provides improved primal solutions compared to a problem-specific integer linear program proposed by \citet{biro2021complexity}.

The remainder of this paper is structured as follows. We formally describe \textsc{Max-Pareto} in Section~\ref{sec:problem}, and prove several structural properties in Section~\ref{sec:properties}. We study Pareto-optimal bipartite matchings in Section~\ref{sec:matchings}, showing that any Pareto-optimal matching also satisfies the stronger notion of fractional Pareto-optimality. Leveraging this result, we prove that \textsc{Max-Pareto} is $\mc{NP}$-complete in Section~\ref{sec:complexity}. Finally. we propose a bilinear programming formulation in Section~\ref{sec:formulation}, and evaluate its computational performance on the problem of finding Pareto-optimal allocations of maximum welfare in Section~\ref{sec:experiments}.

\section{Problem Description}
\label{sec:problem} 

Consider a nonempty, bounded polyhedron $\mc{X} := \{ \bm{x} \in \mathbb{R}^k : A \bm{x} \leq \bm{b} \}$, with $A \in \mathbb{R}^{m \times k}$ and $\bm{b} \in \mathbb{R}^m$. Every solution vector $\bm{x} \in \mc{X}$ maps linearly to an $n$-dimensional payoff vector $\bm{u} := U \bm{x}$, where $n$ is the number of agents under consideration and $U \in \mathbb{R}^{n \times k}$ is a linear mapping from solutions to payoff vectors. For each $i \in \{1, \dots, n\}$, entry $u_i$ equals the payoff that agent $i$ derives from the solution. Without loss of generality, one can replace the linear mapping $U$ by an affine transformation.

Denote by $\mc{U} := \{ U \bm{x} : \bm{x} \in \mc{X} \}$ the set of attainable payoff vectors. We say $\bm{u} \in \mc{U}$ Pareto-dominates $\bm{u'} \in \mc{U}$, denoted by $\bm{u} \succ \bm{u}'$, whenever $u_i \geq u_i'$ for all $i$ and $u_i > u_i'$ for some $i$. In other words, no agent is worse off and at least one agent is strictly better off. We say $\bm{u} \in \mc{U}$ is Pareto-optimal when it is not Pareto-dominated by any $\bm{u}' \in \mc{U}$. Denote the set of Pareto-optimal payoff vectors by $\mc{U}_P \subseteq \mc{U}$, and the corresponding set of Pareto-optimal solutions by $\mc{X}_P := \{ \bm{x} \in \mc{X} : U \bm{x} \in \mc{U}_P \}$. Note that a solution is Pareto-optimal when it generates a Pareto-optimal payoff vector. The goal is to maximise some linear objective function $\bm{c}^\top \bm{x}$ over the set of Pareto-optimal solutions $\mc{X}_P$:
\begin{subequations}
\begin{align}
\max \quad & \bm{c}^\top \bm{x} \\
\text{s.t.} \quad & \bm{x} \in \mc{X}_P. 
\end{align}
\label{eq:max_pareto}
\end{subequations}
We refer to problem~(\ref{eq:max_pareto}) as \textsc{Max-Pareto}.

\section{Structural Properties}
\label{sec:properties}

We now present some structural properties of \textsc{Max-Pareto}. We will use these results to prove $\mc{NP}$-completeness in Section~\ref{sec:complexity} and to develop a bilinear programming formulation in Section~\ref{sec:formulation}.

We start by observing that the set of attainable payoff vectors forms a polyhedron.
\begin{lem}
\label{lem:polyhedron}
The set of attainable payoff vectors $\mc{U}$ is a polyhedron.
\end{lem}
\begin{proof}
Define the auxiliary polyhedron $\mc{Y} := \{ (\bm{x}, \bm{u}) : A \bm{x} \leq b, \bm{u} = U \bm{x} \}$. Since $\mc{U} = \{ \bm{u} : (\bm{x}, \bm{u}) \in \mc{Y}\}$ is a projection of $\mc{Y}$, the set $\mc{U}$ is also a polyhedron.
\end{proof}

It is easy to see that the set of non-dominated payoff vectors lies on the boundary of this polyhedron.
\begin{lem}
The set of Pareto-optimal payoff vectors $\mc{U}_p$ is fully contained in the boundary of $\mc{U}$.
\end{lem}

\begin{proof}
By contradiction, assume there exists an $\bm{u} \in \mc{U}_p$ in the interior of $\mc{U}$. Then, there exists an $\epsilon > 0$ for which the ball of radius $\epsilon$ around $\bm{u}$ is fully contained in $\mc{U}$. In particular, there exists an $\epsilon' > 0$ for which $\bm{u}' := \bm{u} + \epsilon' \bm{1} \in \mc{U}$. Clearly, $\bm{u}' \succ \bm{u}$. This contradicts the fact that $\bm{u} \in \mc{U}_p$. 
\end{proof}

This result is rather intuitive, as Pareto-optimal payoffs must be `maximal' in some sense. In welfare economics, this is formalised through the classic result that Pareto-optimal allocations maximimise a weighted sum of agents' utilities \citep{negishi1960welfare}. We now prove a linear-programming counterpart:
\begin{lem}
\label{lem:support}
A solution $\bm{x} \in \mc{X}$ is Pareto-optimal if and only if there exists a weight vector $\bm{w} \in \mathbb{R}^{n}_{++}$ such that $U \bm{x} \in \arg \max\limits_{\bm{u} \in \mc{U}}  \bm{w}^\top \bm{u}$.
\end{lem}

\begin{proof}
First, let $\bm{w}$ satisfying the conditions be given. Clearly, $\bm{x}$ is Pareto-optimal: domination by some $\bm{x}'$ would imply that $\bm{w}^\top U \bm{x}' > \bm{w}^\top U \bm{x}$, contradicting optimality of $U \bm{x}$.

Conversely, let $\bm{x}$ be Pareto-optimal and let $\bm{u} = U \bm{x}$. As illustrated in Figure~\ref{fig:separation}, we construct a supporting vector $\bm{w}$ based on a hyperplane separating $\mc{U}$ from the set of vectors strictly dominating $\bm{u}$. Define $\mc{T}_{\bm{u}} := \{ \bm{t} \in \mathbb{R}^n : \bm{t} \succeq \bm{u} \} = \{ \bm{u} + \bm{r} : \bm{r} \in \mathbb{R}^{n}_{+}\}$ as the set of points that Pareto-dominate $\bm{u}$, and $\bm{u}$ itself. It is readily seen that $\mc{T}_{\bm{u}}$ is a polyhedron.

\def\nudge{.5}

\tikzset{axis/.style={thick, black, -latex, shorten <=-\nudge cm, shorten >=-2*\nudge cm}}
\tikzset{line/.style={thick, black}}

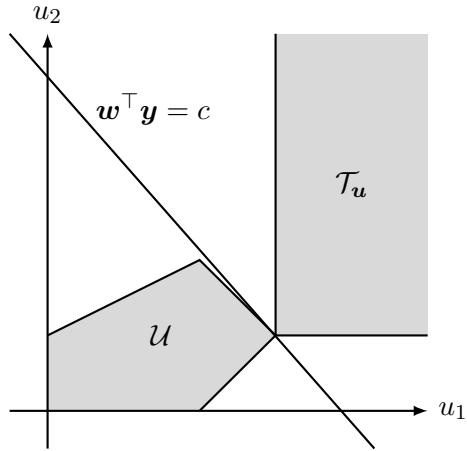
\begin{figure}[htb!]
\centering
\begin{tikzpicture}

\draw[axis] (0,0) -- (4,0) node[right=2* \nudge cm] {\(u_1\)};
\draw[axis] (0,0) -- (0,4) node[above=2*\nudge cm] {\(u_2\)};

\draw[line] (-0.5, 5) -- (4.3,-.5) coordinate (sep);
\node[right] at (0.5, 4) {$\bm{w}^\top \bm{y} = c$};

\begin{scope}
\clip (-\nudge ,-\nudge) rectangle (4+\nudge,4+\nudge);

\draw[thin, line] (0,1) -- (2,2) -- (3,1) -- (2,0);
\clip (0,1) -- (2,2) -- (3,1) -- (2,0) -| (0,1);
\fill[black, opacity=.15] (0,0) rectangle (4,4);
\node at (1 + \nudge, 1) {$\mc{U}$};
\end{scope}

\draw[thin, line] (3,5) -- (3,1) -- (5,1);
\clip (3, 1) -- (3, 5) -- (5, 5) -- (5, 1) -| (3, 1);
\fill[black, opacity=.15] (3, 1) rectangle (5, 5);
\node at (4, 3) {$\mc{T}_{\bm{u}}$};
\end{tikzpicture}
\caption{Illustration of Lemma~\ref{lem:support}.}
\label{fig:separation}
\end{figure}

By Pareto-optimality of $\bm{u}$, it holds that $\mc{U} \cap \mc{T}_{\bm{u}} = \{ \bm{u} \}$. Since $\mc{U}$ is a polyhedron by Lemma~\ref{lem:polyhedron}, there exists a smallest face $F$ of $\mc{U}$ containing $\bm{u}$. Let $G := \{ \bm{t} \in \mathbb{R}^n : \bm{1}^\top \bm{t} = \bm{1}^\top \bm{u} \} \cap \mc{T}_{\bm{u}} = \{\bm{u} \}$ be a smallest face of $\mc{T}_{\bm{u}}$ containing $\bm{u}$. Clearly, $\text{aff}(F \cup G) = \text{aff}(F \cup \{ \bm{u}\}) = \text{aff}(F) \neq \mathbb{R}^n$. By Theorem 3 of \citep{mclennan2002ordinal}, there exists a hyperplane $H := \{ \bm{y} \in \mathbb{R}^n : \bm{w}^\top \bm{y} = c \}$ such that $\mc{U} \subset H^- :=  \{ \bm{y} \in \mathbb{R}^n : \bm{w}^\top \bm{y} \leq c \}$, $\mc{T}_{\bm{u}} \subset H^+ := \{ \bm{y} \in \mathbb{R}^n : \bm{w}^\top \bm{y} \geq c \}$, $\mc{U} \cap H = F$, and $\mc{T}_{\bm{u}} \cap H = G$.

We show that $\bm{w}^\top \in \mathbb{R}^n_{++}$. To the contrary, assume there exists an index $i$ for which $w_i \leq 0$. Consider $\bm{t} := \bm{u} + \bm{e}_i \in \mc{T}_{\bm{u}}$, the payoff vector that Pareto-dominates $\bm{u}$ by one unit on entry $i$. Then $\bm{t}\in H^+$. Since $w_i \leq 0$, it holds that $\bm{w}^\top \bm{t} = \bm{w}^\top (\bm{u} + \bm{e}_i) = c + w_i \leq c$. Thus, $\bm{t}\in H^-$. At the same time, it holds that $\bm{t} \notin G$ and therefore $\bm{t} \notin H$. 
This yields a contradiction, and hence $\bm{w} \in \mathbb{R}^n_{++}$. Since $\mc{U} \subset H^-$ and $U\bm{x}=\bm{u} \in H$, it follows that $U \bm{x} \in \arg \max\limits_{\bm{u} \in \mc{U}} \bm{w}^\top \bm{u}$.
\end{proof}
In the following, we will refer to vectors $\bm{w}$ satisfying the conditions of Lemma~\ref{lem:support} as supporting weight vectors. We will use Lemma~\ref{lem:support} to formulate \textsc{Max-Pareto} as a bilinear program in Section~\ref{sec:formulation}. Moreover, we can use it to prove the existence of optimal vertex solutions to \textsc{Max-Pareto}.

\begin{lem}
\label{lem:opt_vertex}
There always exists an optimal solution to \textsc{Max-Pareto} at a vertex of $\mc{X}$. 
\end{lem}

\begin{proof}
Let $\bm{y}$ be an optimal solution with objective value $z$. Since $\mathcal{X}$ is nonempty and bounded, such a solution always exists. Denote the set of vertices of $\mc{X}$ by $\text{vert}(\mc{X})$. If $\bm{y} \in \text{vert}(\mc{X})$, we are done. Hence, assume the optimal solution is a convex combination of vertices: $\bm{y} = \sum\limits_{\bm{x} \in \text{vert}(\mc{X})} \alpha_{\bm{x}} \bm{x}$ for $\alpha_{\bm{x}} \geq 0, \sum\limits_{\bm{x} \in \text{vert}(\mc{X})} \alpha_{\bm{x}} = 1$. As $\bm{y}$ is Pareto-optimal, by Lemma~\ref{lem:support} there exist a weight vector $\bm{w}\in\mathbb{R}^n_{++}$ and scalar $c$ such that $\bm{w}^\top U \bm{y} = c$ and $\bm{w}^\top U \bm{x} \leq c$ for all $\bm{x} \in \mc{X}$. Consider any vertex $\bm{x}$ for which $\alpha_{\bm{x}} > 0$. Since $\bm{w}^\top U \bm{x} \leq c$ and $\bm{w}^\top U\bm{y} = \sum\limits_{\bm{x} \in \text{vert}(\mc{X})} \alpha_{\bm{x}} \bm{w}^\top U\bm{x}$, it must be that $\bm{w}^\top U \bm{x} = c$. It follows from Lemma~\ref{lem:support} that $\bm{x}$ is Pareto-optimal. 

By optimality of $\bm{y}$, it also holds that $\bm{c}^\top \bm{x} \leq z$. If not, $\bm{x}$ would be a Pareto-optimal solution with strictly higher objective value, contradicting optimality of $\bm{y}$. By the same reasoning as before, it must hold that $\bm{c}^\top \bm{x} = z$. Hence, any $\bm{x}$ for which $\alpha_{\bm{x}} > 0$ is an optimal vertex solution. As $\sum\limits_{\bm{x} \in \text{vert}(\mc{X})} \alpha_{\bm{x}} = 1$, at least one such vertex exists.
\end{proof}

To conclude this section, we identify an interesting case where zero welfare loss is incurred when imposing Pareto-optimality. \citet{biro2021complexity} prove a similar statement in the context of Pareto-optimal allocations and refer to this situation as `aligned interests'. 
\begin{prop}
\label{prop:zero_loss}
Whenever $\bm{c} = U^\top \bm{w}$ for some $\bm{w} \in \mathbb{R}^{n}_{++}$, any optimal solution to $\max\limits_{\bm{x} \in \mc{X}} \{ \bm{c}^\top \bm{x}\}$ is an optimal solution to \textsc{Max-Pareto}.
\end{prop}

\begin{proof}
Let $\bm{c} = U^\top \bm{w}$ for some $\bm{w} \in \mathbb{R}_{++}^{n}$. It holds that 
\begin{displaymath}
\begin{aligned}
\bm{x} \in \arg \max\limits_{\bm{x} \in \mc{X}} \{ \bm{c}^\top \bm{x}\} & \Rightarrow \bm{x} \in \arg \max\limits_{\bm{x} \in \mc{X}} \{(U^\top \bm{w})^\top \bm{x}\} \\
& \Rightarrow \bm{x} \in \arg \max\limits_{\bm{x} \in \mc{X}} \{ \bm{w}^\top U \bm{x}\} \\
& \Rightarrow U \bm{x} \in \arg \max\limits_{\bm{u} \in \mc{U}} \{ \bm{w}^\top \bm{u}\}.
\end{aligned}
\end{displaymath}
This shows that $\bm{x}$ is supported by $\bm{w}$. It follows from Lemma~\ref{lem:support} that $\bm{x}$ is Pareto-optimal.
\end{proof}
A simple example is when $\bm{c}$ is a strictly positive vector and all agents have uniform preferences over the set of feasible solutions, i.e., every row in $U$ contains $n$ identical entries. All instances satisfying the conditions of Proposition~\ref{prop:zero_loss} can be solved in polynomial time using standard linear-programming algorithms.

\section{Pareto-Optimal Bipartite Matchings}
\label{sec:matchings}

In this section, we prove the somewhat surprising result that Pareto-optimal matchings in bipartite graphs always satisfy the stronger notion of fractional Pareto-optimality. This result allows us to model discrete allocation problems as special cases of \textsc{Max-Pareto} and thereby prove $\mc{NP}$-completeness of \textsc{Max-Pareto} in Section~\ref{sec:complexity}. We introduce preliminary notation and definitions in Section~\ref{subsec:matchings_preliminary}. We prove several lemmas on Pareto-optimal matchings in Section~\ref{subsec:matchings_po}, and turn our attention to fractionally Pareto-optimal matchings in Section~\ref{subsec:matchings_fpo}.

\subsection{Preliminaries}
\label{subsec:matchings_preliminary}

Throughout the following, we use $G =  (V_1 \cup V_2, E, \bm{w})$ to denote a weighted bipartite graph from $V_1$ to $V_2$ with edge set $E$ and edge weights $\bm{w} \in \mathbb{R}^{|E|}$. For each subset $R \subseteq V_1 \cup V_2$, the set of vertices incident to a vertex in $R$ is denoted by $N(R)$. For notational convenience, we will use $N(i)$ to indicate $N(\{i\})$. A matching $M \subseteq E$ is a nonempty subset of edges with no common vertices. We use $V_1(M)$ to indicate the nodes in $V_1$ that are matched in $M$. A matching is perfect whenever $V_1(M) = V_1$. For the sake of completeness, we recall Hall's marriage theorem below. This well-known result provides a necessary and sufficient condition for the existence of a perfect matching in bipartite graphs. 
\begin{thm}[Marriage theorem \citep{hall1935representatives}] 
Bipartite graph $G$ admits a perfect matching if and only if $|R| \leq |N(R)|$ for all $R \subseteq V_1$.
\end{thm}
To simplify notation, we use $M(i)$ to denote the vertex in $V_2$ matched to $i \in M(V_1)$ in matching $M$. Similarly, for each subset $R \subseteq M(V_1)$, we define $M(R) = \{M(i): i \in R\}$. Finally, for any subset $R\subseteq V_1$, we define the subgraph $G_R$ that is induced by $V_1\setminus{R}$ and $V_2\setminus M(R \cap V_1(M))$. In words, $G_R$ is the graph that contains the vertices in $V_1\setminus R$, the vertices in $V_2$ that are not equal to $M(i)$ for some $i\in R$, and all edges with both endpoints in these sets. The weight of these edges in $G_R$ is equal to their weight in $G$. Stated differently, the graph $G_R$ is obtained by removing all vertices in $R$ from $V_1$, all vertices matched to a vertex in $R$ from $V_2$, and all edges incident to the removed vertices. This definition is illustrated in Figure~\ref{fig:subgraph}. 

\tikzset{->-/.style={decoration={
			markings,
			mark=at position #1 with {\arrow{>}}},postaction={decorate}}}

\tikzset{mycircle/.style={circle, draw, fill=white, minimum size=7mm, inner sep=0pt}}

\begin{figure}[hbtp]
\centering
\begin{subfigure}[b]{0.45\textwidth}
\centering
\begin{tikzpicture}
    \node[mycircle] at (0, 5) (v1) {$r_1$};
    \node[mycircle] at (0, 4) (v4) {$r_2$};
    \node[mycircle] at (0, 3) (v5) {$v_1$};
    \node[mycircle] at (0, 2) (v6) {$v_2$};
    \node[mycircle] at (0, 1) (v7) {$v_3$};
    
    \node[mycircle] at (3, 4) (v9) {$v_4$};
    \node[mycircle] at (3, 3) (v10) {$v_5$};
    \node[mycircle] at (3, 2) (v11) {$v_6$};
    \node[mycircle] at (3, 1) (v12) {$v_7$};
    \node[mycircle] at (3, 0) (v13) {$v_8$};

    \draw[black] (v4) -- (v9);
    \draw[dashed, black] (v4) -- (v12);
   
    \draw[dashed, black] (v5) -- (v9);
    \draw (v5) -- (v10);
    \draw[dashed,] (v5) -- (v12);
    \draw[dashed ] (v5) -- (v13);
    \draw (v6) -- (v11);
    \draw[dashed] (v6) -- (v12);
    \draw[dashed] (v7) -- (v11);
\end{tikzpicture}
\caption{$G$}
\label{subfig:graph}
\end{subfigure}
\begin{subfigure}[b]{0.45\textwidth}
\centering
\begin{tikzpicture}
    \node[draw=black, fill=white, shape=circle] at (0, 3) (v5) {$v_1$};
    \node[draw=black, fill=white, shape=circle] at (0, 2) (v6) {$v_2$};
    \node[draw=black, fill=white, shape=circle] at (0, 1) (v7) {$v_3$};
    
    \node[draw=black, fill=white, shape=circle] at (3, 3) (v10) {$v_6$};
    \node[draw=black, fill=white, shape=circle] at (3, 2) (v11) {$v_7$};
    \node[draw=black, fill=white, shape=circle] at (3, 1) (v12) {$v_8$};
    \node[draw=black, fill=white, shape=circle] at (3, 0) (v13) {$v_9$};
   
    \draw (v5) -- (v10);
    \draw[dashed,] (v5) -- (v12);
    \draw[dashed ] (v5) -- (v13);
    \draw (v6) -- (v11);
    \draw[dashed] (v6) -- (v12);
    \draw[dashed] (v7) -- (v11);
\end{tikzpicture}
\caption{$G_{R}$}
\label{subfig:reduced_graph}
\end{subfigure}
\caption{Figure~(a) depicts graph $G=(V_1\cup V_2,E)$, with a matching $M$ given by solid lines and other edges of $G$ by dashed lines. Figure~(b) depicts graph $G_R$ for the subset $R=\{r_1,r_2\}\subseteq V_1$.}
\label{fig:subgraph}
\end{figure}
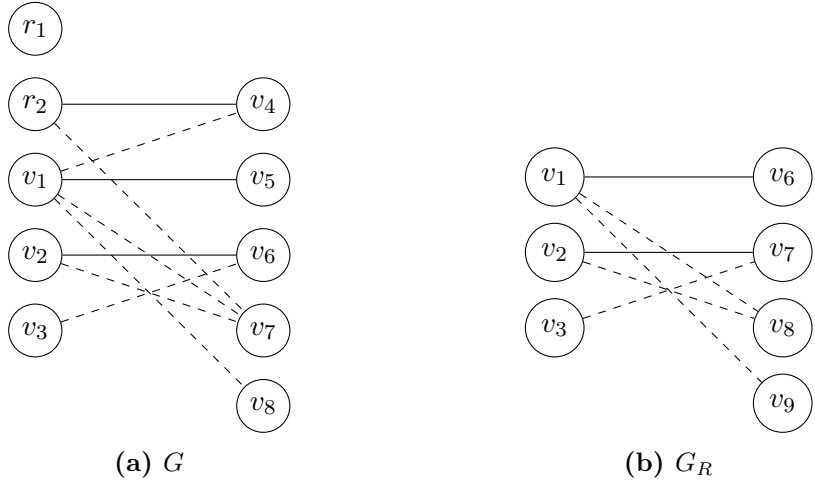

\subsection{Pareto-Optimal Matchings}
\label{subsec:matchings_po}

We now study Pareto-optimal matchings on weighted bipartite graphs. To this end, assume that each matching $M$ maps to a payoff vector $\bm{u}^M \in \mathbb{R}^{|V_1|}$, whose entries are defined as $u_{i}^M = w_{i, M(i)}$ if $i \in V_1(M)$ and $u_{i}^M := 0$ otherwise. In other words, edge weight $w_{i,j}$ represents the payoff to vertex $i \in V_1$ when matched to vertex $j \in V_2$. A matching is Pareto-optimal when there exists no other matching in which at least one vertex receives a strictly higher payoff and no vertex receives a strictly lower payoff.
\begin{defin}
\label{def:po_matching}
Let $\mc{M}$ denote the set of matchings in $G$. We say $M \in \mc{M}$ is Pareto-optimal (PO) when there does not exist an $M' \in \mc{M}$ for which $\bm{u}^{M'} \succ \bm{u}^{M}$.
\end{defin}

The following simple lemma will prove useful.
\begin{lem}\label{lem:submatrix}
Let matching $M$ in $G$ be PO. For any nonempty $R \subset V_1(M)$, the matching $M' := \{ (i, M(i)) : i \in V_1(M)\setminus R\}$ in the reduced graph $G_R$ is PO. 
\end{lem}
\begin{proof}
Suppose, to the contrary, that matching $M'$ in $G_{R}$ is dominated by matching $M''$. Augment $M''$ to a matching in $G$ by adding all edges in $M \setminus M'$. This augmented matching dominates $M$, contradicting Pareto-optimality of $M$.
\end{proof}
The proof above is illustrated in Figure~\ref{fig:subgraph}. Assume $M$ is Pareto-optimal. The matching $M'$ is given by the solid lines in Figure~\ref{subfig:reduced_graph}. If this matching $M'$ in $G_R$ would be dominated by a matching $M''$ in $G_R$, then $M''$ can be augmented to find a matching in $G$ that dominates $M$. 

\begin{defin}
\label{def:blocking}
Let $M$ be a matching in $G$. We say that nonempty subset $B \subseteq V_1(M)$ is a blocking set with respect to $M$ when it satisfies the following conditions:
\begin{enumerate}[(i)]
    \item $w_{i,j} \leq w_{i, M(i)}$ for all $i \in B, j \in N(i)$, 
    \item $w_{i,j} < w_{i,M(i)}$ for all $i \in B, j \in N(i) \setminus M(B)$.
\end{enumerate}
\end{defin}
Informally, a blocking set is a set of vertices $B\subseteq V_1(M)$ that all (i) receive their highest possible payoff and (ii) are strictly worse off when matched to any node $j \notin M(B)$. As a consequence, if another matching $M'$ matches any node $i\in V_1\setminus B$ to a node $j$ in $M(B)$, then at least one node in $B$ will be strictly worse off in $M'$ than in $M$. We provide an example below.
\begin{ex}
\label{ex:blocking}
Consider the weighted bipartite graph in Figure~\ref{fig:bipartite}. The non-dominated matchings in this graph are $M_1 := \{ (v_2, v_5), (v_3, v_6)\}$, $M_2 := \{ (v_1, v_5), (v_2, v_4), (v_3, v_6)\}$, and $M_3 := \{ (v_1, v_6), (v_2, v_5), (v_3, v_4)\}$, with corresponding payoff vectors $\bm{u}^1 = (0, 2, 4)$, $\bm{u}^2 = (1, 1, 4)$, and $\bm{u}^3 = (2, 2, 2)$. Blocking sets with respect to matching $M_1$ are $\{ v_2 \}$, $\{v_3\}$, and $\{v_2, v_3\}$.

\tikzset{->-/.style={decoration={
			markings,
			mark=at position #1 with {\arrow{>}}},postaction={decorate}}}

\tikzset{mycircle/.style={circle, draw, fill=white, minimum size=7mm, inner sep=0pt}}

\begin{figure}[hbtp]
\centering
\begin{tikzpicture}
	
    \node[mycircle] at (0, 3) (v1) {$v_1$};
    \node[mycircle] at (0, 1.5) (v2) {$v_2$};
    \node[mycircle] at (0, 0) (v3) {$v_3$};

    \node[mycircle] at (3, 3) (v4) {$v_4$};
    \node[mycircle] at (3, 1.5) (v5) {$v_5$};
    \node[mycircle] at (3, 0) (v6) {$v_6$};

    \draw[dashed] (v1) -- node[pos=0.15, above] {\scriptsize{1}} (v5);
    \draw[dashed] (v1) -- node[pos=0.05, below] {\scriptsize{2}} (v6);
    \draw[dashed] (v2) -- node[pos=0.1, above] {\scriptsize{1}} (v4);
    \draw (v2) -- node[pos=0.1, below] {\scriptsize{2}} (v5);
    \draw[dashed] (v3) -- node[pos=0.1, above] {\scriptsize{2}} (v4);
    \draw (v3) -- node[pos=0.05, below] {\scriptsize{4}} (v6);

\end{tikzpicture}
\caption{Bipartite graph of Example~\ref{ex:blocking}. Edges included in matching $M_1$ are solid lines, the remaining edges are dashed lines.}
\label{fig:bipartite}
\end{figure}
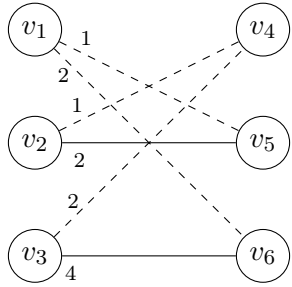
\end{ex}

It is easy to see that the union of blocking sets is again a blocking set. Moreover, a blocking set always exists for Pareto-optimal matchings in which at least one vertex is not assigned its highest possible payoff.
\begin{lem}
Let $M$ in $G$ be PO. If $w_{i,j} > u_i^M$ for some $i \in V_1$ and $j \in N(i)$, then there exists a blocking set with respect to $M$. \label{lem:blocking}
\end{lem}
\begin{proof}
We use induction on $|V_1|$. First, we show that the statement holds for $|V_1| = 2$. In that case, $V_1 = \{i, i'\}$ for some node $i'\in V_1$ and $j\in V_2$. We know that $G$ contains edge $(i, j)$, for which it holds that $w_{i,j} > u_i^M$. Moreover, both the graph $G$ and $M$ must contain edge $(i', j)$ with weight $w_{i', j} > 0$. If not, the matching $M'=\{(i,j)\}$ dominates $M$. Suppose first that there are no other edges incident to $i'$. Then we immediately find that $\{ i' \}$ is a blocking set. Then suppose that $(i', l)$ is included in $G$ for some $l\in V_2\setminus\{j\}$. To show that $\{ i' \}$ still forms a blocking set, we must show that $w_{i', l} < w_{i', j}$. To the contrary, suppose that $w_{i', l} \geq w_{i', j}$. Then matching $M' := \{ (i, j), (i', l) \}$ would Pareto-dominate $M$. We conclude again that $\{ i' \}$ is a blocking set, and hence the statement holds for $|V_1| = 2$. 

Next, assume the statement holds for all $|V_1| \leq k$ and consider a graph $G$ with $|V_1| = k + 1$ satisfying the conditions. Then $M$ is a PO matching and $i\in V_1$, $j\in V_2$ are such that $w_{i,j} > u_i^M$. 

In order to prove the existence of a blocking set, consider the auxiliary graph $G'=(V_1' \cup V_2, E')$, where $V_1'=V_1(M)\cup\{i\}$ and 
$$E'=\{(k,l)\in E: k\in V_1'\setminus\{i\}, w_e \geq w_{k,M(k)}\}\cup\{(i,j)\}.$$
Observe that $(k,M(k))\in E'$ if $k\neq i$. By construction, any perfect matching in $G'$ dominates $M$, because the payoff for $i$ strictly increases to $w_{ij}$ and the payoff for all other vertices in $V_1(M)$ does not decrease. Thus, by Pareto-optimality of $M$, there cannot be a perfect matching in $G'$. By Hall's marriage theorem, there must exist a set $R\subseteq V_1'$ such that $|N(R)|<|R|$. We will use $R$ to construct a blocking set in $G$ with respect to $M$. Figure~\ref{fig:proof} depicts part of $G'$ in which $R$ and $N(R)$ are indicated.

Denote $R'=R\setminus\{i\}$. Observe first that $M(k)\in N(k)$ for all $k\in R'$ and that all $M(k)$ are different. If $i\notin R$, then $|N(R)|=|N(R')|\geq |M(R')|=|R'|=|R|$. We conclude that $i\in R$. Pareto-optimality of $M$ implies the existence of a node $i'\in V_1'$ such that $j=M(i')$. If $i'\notin R$, then $N(R)\supseteq M(R')\cup\{j\}$ and $j\notin M(R')$, so again $|N(R)|\geq |R|$. We conclude that $i'\in R$ as well. Now if there would be an $l\in V_2\setminus M(R)$ such that $w_{k,l} \geq w_{k,M(k)}$, then the edge $(k,l)\in E'$ and again $|N(R)|\geq |R|$. It follows that no such edge $(k,l)$ is in $E'$, and thus that all edges $(k,l)\in E$ with $k\in R$ and $l\in V_2\setminus M(R)$ satisfy $w_{k,l}<w_{k,M(k)}$. There are now two possibilities:
\begin{enumerate}[(i)]
        \item  It holds for all $k\in R'$ and $l\in M(R')$ that $w_{k,l}\leq w_{k,M(k)}$. In that case, $R'$ is a blocking set.
        \item There is an edge $(k,l)$ with $k\in R'$ and $l\in M(R')$ such that $w_{k,l} > w_{k,M(k)}$. Define the matching $M'$ as $M$ if $i$ is not matched in $M$, and as $M\setminus \{(i,M(i)\}$ otherwise. By Lemma~\ref{lem:submatrix}, the matching $M'$ is Pareto-optimal in $G_{\{i\}}$, and the edge $(k,l)$ satisfies $w_{k,l} > w_{k,(M(k)} = u_k^M$. By the induction hypothesis, there is a blocking set $B$ in $G_{\{i\}}$ with respect to $M'$. If $i$ is not matched in $M$, then $B$ is easily seen to also be a blocking set in $G$ with respect to $M$. Otherwise, recall that for all $k\in R'$ and $l\in V_2\setminus M(R')$, we have $w_{k,l}<w_{k,M(k)}$. This holds particularly for $l=M(i)$. This demonstrates that $B\cap R'$ is a blocking set in $G$ with respect to $M$.
\end{enumerate}
We conclude that a blocking set exists in both cases. The proof follows from induction.
\end{proof}

\tikzset{->-/.style={decoration={
			markings,
			mark=at position #1 with {\arrow{>}}},postaction={decorate}}}

\tikzset{mycircle/.style={circle, draw, fill=white, minimum size=7mm, inner sep=0pt}}

\begin{figure}[hbtp]
\centering
\begin{subfigure}[t]{0.45\textwidth}
\centering
\begin{tikzpicture}
    \node[mycircle] at (0, 7) (v1) {$i$};
    \node[mycircle] at (0, 6) (v2) {$i'$};
    \node[mycircle] at (0, 5) (v3) {$k$};
    \node at (0, 4) (v4) {$\vdots$};
    \node[mycircle] at (0, 3) (v5) {$k'$};
    
    \node[mycircle, densely dotted] at (3, 7) (v8) {};
    \node at (4, 7) (v120) {$M(i)$};
    \node[mycircle] at (3, 6) (v9) {$j$};
    \node[draw=black, fill=white, shape=circle] at (3, 5) (v10) {\phantom{$k$}};
    \node at (4, 5) (v120) {$M(k)$};
    \node[fill=white, shape=circle] at (3, 4) (v11) {$\vdots$};
    \node[mycircle] at (3, 3) (v12) {};
    \node at (4, 3) (v120) {$M(k')$};
    \node[mycircle] at (3, 8) (v13) {$l$};
    \draw[dotted] (v1) -- (v8);
    \draw[dashed] (v1) -- (v9);
    \draw (v2) -- (v9);
    \draw (v3) -- (v10);
    \draw (v5) -- (v12);

    \phantom{\node[draw,dotted,rounded corners,fit=(v9)(v120),inner sep=3mm] (N) {};}
\end{tikzpicture}
\caption{$G$}
\end{subfigure}
\begin{subfigure}[t]{0.45\textwidth}
\centering
\begin{tikzpicture}
    \node[draw=black, fill=white, shape=circle] at (8, 7) (v1) {$i$};
    \node[draw=black, fill=white, shape=circle] at (8, 6) (v2) {$i'$};
    \node[draw=black, fill=white, shape=circle] at (8, 5) (v3) {$k$};
    \node[fill=white, shape=circle] at (8, 4) (v4) {$\vdots$};
    \node[draw=black, fill=white, shape=circle] at (8, 3) (v5) {$k'$};
    
    \node[mycircle] at (11, 6) (v9) {$j$};
    \node[mycircle] at (11, 5) (v10) {};
    \node[] at (12, 5) (v120) {$M(k)$};
    \node[fill=white, shape=circle] at (11, 4) (v11) {$\vdots$};
    \node[mycircle] at (11, 3) (v12) {};
    \node[] at (12, 3) (v130) {$M(k')$};
    \node[mycircle] at (11, 8) (v13) {$l$};
    \draw (v1) -- (v9);
    \draw (v2) -- (v9);
    \draw (v2) -- (v12);
    \draw (v3) -- (v9);
    \draw (v3) -- (v10);
    \draw (v5) -- (v9);
    \draw (v5) -- (v10);
    \draw (v5) -- (v12);

    \node[draw,dotted,rounded corners,fit=(v1)(v5),inner sep=3mm] (R) {};
    \node[anchor=south east] at (R.north east) {$R$};
    \node[draw,dotted,rounded corners,fit=(v9)(v130),inner sep=3mm] (N) {};
    \node[anchor=south east] at (N.north east) {$N(R)$};
\end{tikzpicture}
\caption{$G'$}
\end{subfigure}
\caption{Figure~(a) depicts some vertices of the graph $G=(V_1\cup V_2,E)$, as well as a matching $M$ in $G$ with solid lines. Moreover, the edge $(i,j)\in E$ is depicted by a dashed line. The vertex $M(i)$ and the edge $(i,M(i))$, depicted by a dotted circle and line, are not necessarily included in the graph, depending on whether $i$ is matched or not. Figure~(b) depicts the resulting graph $G'$, indicating the sets $R$ and $N(R)$.}
\label{fig:proof}
\end{figure}
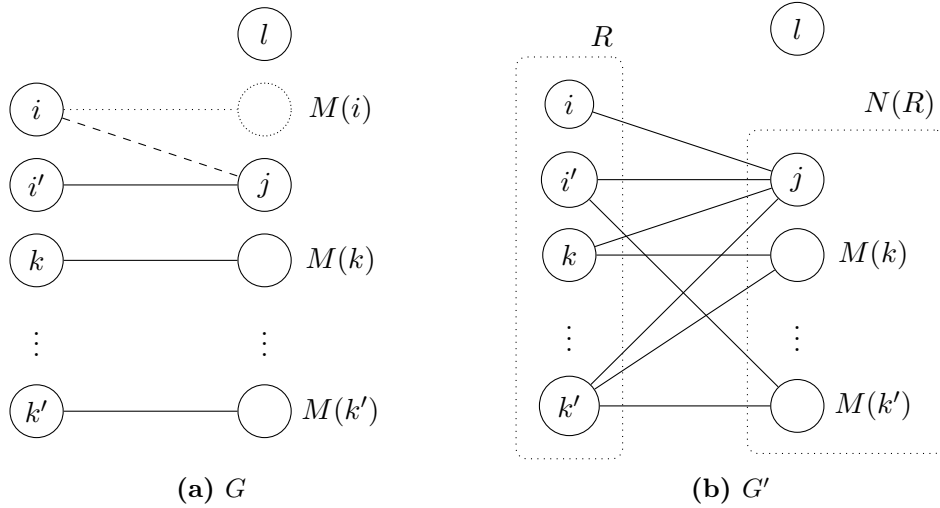

\subsection{Fractionally Pareto-Optimal Matchings}
\label{subsec:matchings_fpo}

We now turn our attention to the stronger notion of fractional Pareto-optimality. To illustrate the concept of fractional Pareto-optimality, consider three Pareto-optimal payoff vectors $(5,1)$, $(2,2)$, and $(1,5)$. It holds that $(2,2)$ is Pareto-dominated by the convex combination $\frac12(5,1)+\frac12(1,5)=(3,3)$. Similarly, a matching is fractionally Pareto-optimal when it is not Pareto-dominated by any convex combination of matchings.

\begin{defin}
\label{def:fpo_matching}
Let $\mc{M}$ denote the set of matchings of $G$. We say $M \in \mc{M}$ is fractionally Pareto-optimal (fPO) when there does not exist a vector $\bm{\alpha} \in \mathbb{R}^{|\mc{M}|}_{+}$ satisfying $\bm{1}^\top \bm{\alpha} = 1$ and $\sum\limits_{M' \in \mc{M}} \alpha_{M'} \bm{u}^{M'} \succ \bm{u}^{M}$.
\end{defin}
Clearly, every fPO matching is PO. Perhaps surprisingly, every PO matching is also fPO.
\begin{thm}
If matching $M$ in $G$ is PO, then $M$ is fPO. 
\label{thm:fpo_matching}
\end{thm}

\begin{proof}
By contradiction, let $G' := (V_1' \cup V_2', E', \bm{w}')$ be a graph that contains a matching $M$ that is PO but not fPO with the smallest possible set $|V_1|$. Then, there exist a set $\mc{N}$ of matchings in $G'$ and vector $\bm{\alpha} \in \mathbb{R}^{|\mc{N}|}_{+}$ satisfying $\bm{\alpha}^\top \bm{1} = 1$ and $\sum\limits_{M' \in \mc{N}} \alpha_{M'} \bm{u}^{N} \succ \bm{u}^{M}$. Let $i \in V_1'$ be a vertex for which $\sum\limits_{M' \in \mc{M}} \alpha_{M'} u_{i}^{M'} > u_{i}^{M}$.

This implies the existence of a matching $M' \in \mc{M}$ and vertex $j \in V_2'$ satisfying $(i, j) \in M'$ and $w'_{i,j} > u_i^M$. By Lemma~\ref{lem:blocking}, there exists a blocking set $B \subseteq V'_1(M)$. By definition, for every $k \in B$ it holds that $w'_{k,l} \leq w_{k,M(k)}$ for all $l\in N(k)$ and $w'_{k,l} < u_k^M$ for every $l \in N(K)\setminus M(B)$, and thus every $k \in B$ can only be matched to some $l \in M(B)$ in every matching $M' \in \mc{M}$ for which $\alpha_{M'} > 0$.
Thus, $M'(B) = M(B)$ for all matchings $M'\in \mc{M}$ with $\alpha_{M'} > 0$. This also proves that $i\notin B$ and $j \notin M(B)$. Hence, we can construct a strictly smaller counterexample on $G'_{R}$. This shows that no smallest graph satisfying the conditions can exist, contradicting our initial assumption. We conclude that $M$ is fPO.
\end{proof}

\section{Complexity}
\label{sec:complexity}

We prove $\mathcal{NP}$-hardness of \textsc{Max-Pareto} through a reduction from \textsc{Constrained Pareto-Optimal Matching}, which was shown to be $\mathcal{NP}$-hard by \citet{saban2015complexity}. The goal of this problem is to determine whether a bipartite graph admits a Pareto-optimal matching in which a specific subset of objects is matched. To reduce this problem to \textsc{Max-Pareto}, we exploit that (i) bipartite matching can be formulated as a linear program with a totally unimodular constraint matrix and (ii) PO=fPO for bipartite matchings by Theorem~\ref{thm:fpo_matching}. 

We introduce some additional notation to formally define \textsc{Constrained Pareto-Optimal Matching}. Given a set of $n$ agents $\mc{A}$ and objects $\mc{O}$, a preference profile $P := (P_1, \dots, P_n)$ is a collection of strict orderings of objects. In particular, object $a$ preceding object $b$ in ordering $P_i$, denoted by $a \succ_{P_i} b$, indicates that agent $i$ prefers object $a$ over $b$. Profile $P_i$ need not be a complete ordering. In that case, objects not included in the ordering are said to be inadmissible to agent $i$. 

\textsc{Constrained Pareto-Optimal Matching} can now be stated as follows: Given a preference profile $P$ from agents $\mc{A}$ over objects $\mc{O}$, in which agents may have inadmissible objects, and a subset of objects $Q \subseteq \mc{O}$, does there exist a Pareto-optimal matching in which all objects in $Q$ are matched?

A reduction from the above problem allows us to prove the following hardness result.
\begin{thm}
The decision version of \textsc{Max-Pareto} is $\mc{NP}$-complete.
\end{thm}

\begin{proof}
It is readily seen that \textsc{Max-Pareto} is in $\mc{NP}$, as membership of $\bm{x} \in \mathcal{X}$ can be easily confirmed, and Pareto-optimality can be checked in polynomial time by solving the following linear program:
\begin{subequations}
\begin{align}
\max \quad & \bm{e}^\top U \bm{y} \\
\text{s.t.} \quad & U \bm{y} \geq U \bm{x} \\ 
& \bm{y} \in \mc{X}.  
\end{align}
\end{subequations}
We now reduce \textsc{Constrained Pareto-Optimal Matching} to an instance of the decision version of \textsc{Max-Pareto}: Does there exist an $\bm{x} \in \mc{X}_{P}$ for which $\bm{c}^\top \bm{x} \geq k$? We construct this instance of $\textsc{Max-Pareto}$ as follows.
\begin{itemize}
    \item \textit{Feasible region:} Construct a bipartite graph $G := (\mc{A} \cup \mc{O}, E, \bm{w})$ as follows. Edge $e = (i, a)$ is included in $E$ if agent $i \in \mc{A}$ admits item $a \in \mc{O}$, and has weight $w_{(i, a)} := 1 + | \{ o \in \mc{O} : a \succ_{P_i} o\} |$. In other words, we choose weights that assign more value to objects higher up in the agent's preference profile. The bipartite graph is not necessarily complete. We now define $\mc{X} := \{ \bm{x} \in \mathbb{R}^{|E|}_{+} : \sum\limits_{e \in \delta(v)} x_e \leq 1, \forall v \in V_1 \cup V_2\}$ as the matching polytope of $G$, where $\delta(v)$ is the set of edges incident to $v$. Since $G$ is bipartite, the corresponding constraint matrix is totally unimodular and each vertex of the matching polytope corresponds to an integral matching.
    \item \textit{Payoff mapping:} There are $|\mc{A}|$ agents. For any $e = (i, a) \in E$, we set $U_{(i, e)} := w_e$. This payoff function ensures that an integral matching is Pareto-optimal in \textsc{Max-Pareto} if and only if it is Pareto-optimal in \textsc{Constrained Pareto-Optimal Matching}.
    \item \textit{Objective function:} Choose $\bm{c} \in \mbb{R}^{|E|}$ by setting $c_{(i, a)} := 1$ if $a \in Q$ and $c_{(i, a)} := 0$ otherwise. The objective is thus to maximise the number of matched objects in $Q$.
    \item \textit{Objective threshold:} The objective threshold equals $k := |Q|$.
\end{itemize} 
This instance of \textsc{Max-Pareto} is constructed in time polynomial in the size of the input. We now show that it is a yes-instance if and only if \textsc{Constrained Pareto-Optimal Matching} is a yes-instance.

$\Rightarrow$ First, suppose that \textsc{Constrained Pareto-Optimal Matching} is a yes-instance. Let $M$ be a Pareto-optimal matching containing all items in $Q$. Choose $\bm{x}$ by setting $x_{e} := 1$ if $e \in M$, and $x_e := 0$ otherwise. It is clear that $\bm{x}$ is feasible in $\mc{X}$ and meets the objective threshold. 

It remains to show that $\bm{x}$ satisfies Pareto-optimality in \textsc{Max-Pareto}. By construction, this is equivalent to showing that $M$ is fPO with respect to all matchings in $G$. This follows directly from Theorem~\ref{thm:fpo_matching} and the fact that $M$ is PO. We conclude that \textsc{Max-Pareto} is a yes-instance.

$\Leftarrow$ Now, suppose that \textsc{Max-Pareto} is a yes-instance with corresponding solution $\bm{x}$. We aim to convert this to a Pareto-optimal matching $M$ that matches all items in $Q$. By Lemma~\ref{lem:opt_vertex}, we know there exists an optimal vertex solution $\bm{x}'$. By total unimodularity of the bipartite matching polytope, this solution corresponds to an integral matching. This matching is not Pareto-dominated by any point in $\mc{X}$. In particular, it is not dominated by any integral matching. Since $\bm{c}^\top \bm{x} \geq k$, the matching contains all objects in $Q$. Hence, \textsc{Constrained Pareto-Optimal Matching} is a yes-instance.  

Since \textsc{Constrained Pareto-Optimal Matching} is $\mc{NP}$-hard and can be reduced to \textsc{Max-Pareto} in polynomial time, we conclude that \textsc{Max-Pareto} is $\mc{NP}$-hard as well.
\end{proof}
The structure of our proof shows that a large class of discrete allocation problems with Pareto-optimality constraints, such as that of \citep{biro2021complexity}, can be seen as special cases of \textsc{Max-Pareto}.

\section{Bilinear Programming Formulation}
\label{sec:formulation}

In this section, we provide a mathematical programming formulation for \textsc{Max-Pareto}. It is not trivial to express the Pareto-optimality condition, but Lemma~\ref{lem:support} provides a way forward. In particular, it allows us to formulate Pareto-optimality as a series of bilinear constraints.\footnote{Like \textsc{Max-Pareto}, linear programming with bilinear constraints is generally $\mc{NP}$-hard \citep{bennett1993bilinear}.} The bilinear programming formulation of \textsc{Max-Pareto} reads as follows: \begin{subequations}
\begin{align}
\max \quad & \bm{c}^\top \bm{x} \label{eq:obj} \\
\text{s.t.} \quad& \bm{u} = U \bm{x} \label{eq:u} \\
& \bm{w}^{\top} \bm{u} \geq  \bm{w}^\top \bm{u}' && \forall \bm{u}' \in \mc{U} \label{eq:pareto} \\  
& \bm{w} \geq \bm{1} \label{eq:domain_w} \\
& \bm{x} \in \mc{X} \label{eq:domain_x}. 
\end{align}
\label{eq:bilinear}
\end{subequations}
The objective function is given by~(\ref{eq:obj}). Constraints (\ref{eq:u}) ensure that $\bm{u}$ is a feasible payoff vector, while bilinear constraints (\ref{eq:pareto}) guarantee that $\bm{u}$ is supported by weight vector $\bm{w}$. Constraints~(\ref{eq:domain_w}) ensure that $\bm{w}$ meets the remaining conditions of Lemma~\ref{lem:support}. The feasible domain of $\bm{x}$ is given by (\ref{eq:domain_x}). It follows that $\bm{x} \in \mc{X}_P$ for any solution $(\bm{x}, \bm{u}, \bm{w})$ to (\ref{eq:bilinear}).

Bilinear problems can be solved to global optimality by modern-day commercial solvers like Gurobi. The exponential number of bilinear Pareto-optimality constraints (\ref{eq:pareto}) remains troublesome, however. A potential solution would be a cutting plane approach, where constraints (\ref{eq:pareto}) are dynamically separated. 

Alternatively, we can use strong LP duality to obtain a tractable reformulation containing only a single bilinear constraint. Since $\mc{U} = \{ U \bm{x} : \bm{x} \in \mc{X}\}$, constraints (\ref{eq:pareto}) are equivalent to 
\begin{subequations}
\begin{align}
\bm{w}^\top \bm{u} & \geq \bm{w}^\top \bm{U \bm{x}} && \forall \bm{x} \in \mc{X} \label{eq:reformulate_a} \\
\iff \bm{w}^\top \bm{u} & \geq \max_{\bm{x} \in \mc{X}} \{ \bm{w}^\top U \bm{x} \} \label{eq:reformulate_b} \\ 
\iff \bm{w}^\top \bm{u} & \geq \max_{\bm{x} \in \mathbb{R}^k} \{ \bm{w}^\top U \bm{x} : A \bm{x} \leq \bm{b} \}.  \label{eq:reformulate_c}
\end{align}
\end{subequations}
Since $\mc{X}$ is nonempty and bounded, we can apply strong linear programming duality. Treating $\bm{w}$ as fixed, we obtain that (\ref{eq:reformulate_c}) is equivalent to 
\begin{align}
\bm{w}^{\top} \bm{u} \geq \min_{\bm{\eta} \in \mathbb{R}^{m}_{+}} \{ \bm{b}^\top \bm{\eta} : A^\top \bm{\eta} = U^\top \bm{w} \}. \label{eq:reformulate_d}
\end{align}
Finally, note that (\ref{eq:reformulate_d}) can be satisfied if and only if there exists a solution to the following system of equations:
\begin{subequations}
\begin{align}
&\bm{w}^\top \bm{u} \geq \bm{b}^\top \bm{\eta} \label{eq:cons_bilinear} \\
& A^\top \bm{\eta} = U^\top \bm{w} \\ 
& \bm{\eta} \geq \bm{0}. 
\end{align}
\label{eq:reformulate_e}
\end{subequations}
To conclude, replacing (\ref{eq:pareto}) by (\ref{eq:reformulate_e}) yields a tractable reformulation with exactly one constraint containing $n$ bilinear terms. Preliminary computational experiments show that this indeed outperforms a cutting plane approach.

\subsection{Bounded Weight Heuristic}

Ideally, one would like to place a reasonable upper bound $\bm{w} \leq \bm{\bar{w}}$ on the weight vector entries. This would strengthen the formulation, and, most importantly, help avoid numerical issues where excessively large weight vectors are chosen to exploit the feasibility tolerance of a solver. Unfortunately, we are unable to derive a meaningful valid upper bound on $\bm{w}$ that preserves global optimality. As the following result shows, the magnitude of weight vector entries required to support a Pareto-optimal point can grow exponentially in the input size. 

\begin{prop}
For every integer $n \geq 2$, there exist a complete weighted bipartite graph $G$ with $2n$ nodes and weights in $[0, n]$, and a Pareto-optimal matching $M$ in $G$ that can only be supported by a weight vector $\bm{w}$ satisfying $\frac{w_1}{w_n} \geq (n-1)^{n-1}$. 
\end{prop}

\begin{proof}
Take a complete bipartite graph from $n$ to $n$ nodes with edge weights defined as $e_{ij} = 1$ if $i=j$, $e_{ij} = n$ if $i = j + 1$, and $e_{ij} = 0$ otherwise. The all-diagonal matching, including edges $(i, i)$ for all $i$, has payoff vector $\bm{u} = (1, \dots, 1)$ and is Pareto-optimal. As such, there exists a weight vector $\bm{w} \in \mathbb{R}_{++}^{n}$ satisfying $\bm{w}^\top \bm{u} \geq \bm{w}^\top \bm{v}$ for all alternative payoff vectors $\bm{v}$. In particular, it holds for all payoff vectors $\bm{v}_k = (1, \dots, 1, 0, n, 1, \dots, 1)$, $k \geq 2$, obtained by swapping edges $(k - 1, k - 1)$ and $(k, k)$ in the all-diagonal matching with $(k - 1, k)$ and $(k, k - 1)$. As such, $\bm{w}$ must satisfy 
\begin{align}
    \sum_{i=1}^{n} w_i \geq \sum_{i=1}^{k - 2} w_i + n w_k + \sum_{i=k+1}^{n} w_i
\end{align}
for all $k \geq 2$. Combining these inequalities yields $w_i \geq (n-1) w_{i+1}$ for all $i=1,\ldots,n-1$. We conclude that $w_1 \geq (n-1)^{n-1} w_n$. 
\end{proof}

Despite this negative result, our bilinear program can be used as a primal heuristic by still enforcing an upper bound on $\bm{w}$. This restricts the feasible region to a subset of the Pareto-frontier that is still supported by the set of bounded weight vectors. However, one no longer obtains a valid upper bound to \textsc{Max-Pareto}. Even when the resulting weight vector is not at bound, one cannot conclude global optimality due to non-convexity of the bilinear constraint (\ref{eq:cons_bilinear}).

Of course, more sophisticated heuristics for \textsc{Max-Pareto} can be devised. For example, one could exploit that, for any fixed weight vector, feasible, Pareto-optimal solutions can be obtained by solving an ordinary linear program. This potentially paves the way for meta-heuristics that explore the space of weight vectors in an efficient manner.

\section{Computational Experiments}
\label{sec:experiments}

We evaluate the computational performance of our bilinear programming formulation on the $\mc{NP}$-hard problem of finding Pareto-optimal allocations of highest welfare as introduced by \citet{biro2021complexity}. Consider a set of agents $\mc{A}$ and items $\mc{I}$. Allocating item $i \in \mc{I}$ to agent $a \in \mc{A}$ yields welfare $w_{ia}$ and a payoff $u_{ia}$ to agent $a$. We can allocate at most one item to each agent and each item to at most one agent. An allocation is Pareto-optimal if no agent can achieve a strictly higher payoff without making at least one other agent strictly worse off. Our goal is to find a Pareto-optimal allocation of highest welfare.

To model this problem using the bilinear programming formulation of Section~\ref{sec:formulation}, we set $\mc{X}$ equal to the bipartite matching polytope. We benchmark this approach against the integer linear programming (ILP) model presented by \citet{biro2021complexity}. This model is based on the observations that (i) an allocation is Pareto-optimal if and only if the allocation is supported by a price vector at competitive equilibrium and (ii) the conditions for such an equilibrium can be written as a set of linear constraints. 

We generate random instances with number of agents $|\mc{A}| \in \{10, 25, 50, 75, 100\}$ and a number of items $|\mc{I}|$ equal to either 1, 2, 5, or 10 times the number of agents. We assume that all items are admissible to all agents. The welfare $w_{ia}$ and payoff $u_{ia}$ of allocating item $i$ to agent $a$ are randomly sampled from $\{1, \dots, |\mc{I}|\}$. Both models are implemented in Java and solved using Gurobi 12.03 with a time limit of ten minutes. In the bilinear programming formulation, we experiment with three different values for the maximum weight vector entries, using $\bm{\bar{w}} \in \{|\mc{I}|/2, |\mc{I}|, 2 |\mc{I}|\}$. Note that, as a result, our method is a heuristic and no longer provides valid upper bounds. All experiments are conducted on a personal computer with a 2.40Ghz Intel Core i7-13700H processor and 32GB RAM. 

Table~\ref{table:results} reports the results of our computational experiments. For each instance and formulation, it reports the objective value of the best found feasible solution (LB), an upper bound on the objective value (UB), and the computing time. A $*$ indicates that the time limit of ten minutes was reached. For each instance, the best found feasible solution is indicated in bold. 

\begin{table}[htbp!]
  	\centering 
  	\renewcommand{\arraystretch}{1}
  	\caption{Computational results. The best found solution of each instance is indicated in bold. A $*$ indicates the time limit was reached.}
    \label{table:results}
  	\begin{adjustbox}{max width= \textwidth}
  	\begin{tabular}{@{\extracolsep{0pt}}llrrrrrrrrrrrr} 
		\toprule 
		& & \multicolumn{3}{c}{ILP} & \multicolumn{3}{c}{Bilinear ($\bar{w} = |\mc{I}|/2$)} & \multicolumn{3}{c}{Bilinear ($\bar{w} = |\mc{I}|$)} & \multicolumn{3}{c}{Bilinear ($\bar{w} = 2|\mc{I}| $)} \\
		\cmidrule(l){3-5} \cmidrule(l){6-8} \cmidrule(l){9-11} \cmidrule(l){12-14} 
        Agents & Items & LB & UB & Time (s) & LB & UB & Time (s) & LB & UB & Time (s) & LB & UB & Time (s) \\ 
        \midrule
        \multirow{4}{*}{10} & 10 & \textbf{69} & 69 & 0 & 66 & 66.3 & $*$ & 67 & 67 & 53 & \textbf{69} & 69 & 154 \\ 
		& 20 & \textbf{142} & 142 & 0 & \textbf{142} & 142 & 9 & \textbf{142} & 142 & 0 & \textbf{142} & 142 & 2 \\
        & 50 & \textbf{233} & 233 & 0 & 224 & 224 & 0 & 224 & 224 & 0 & 224 & 224 & 0 \\
        & 100 & \textbf{540} & 540 & 0 & \textbf{540} & 540 & 0 & \textbf{540} & 540 & 0 & \textbf{540} & 540 & 0 \\ 
        \midrule
		\multirow{4}{*}{25} & 25 & \textbf{498} & 498 & 96 & 463 & 545.8 & $*$ & 459 & 558.0 & $*$ & 472 & 569.4 & $*$ \\ 
		& 50 & \textbf{651} & 959 & $*$ & 638 & 638 & 226 & 638 & 638 & 2 & 638 & 638 & 4 \\
        & 125 & \textbf{1,635} & 2,869 & $*$ & 1,615 & 1,615 & 0 & 1,615 & 1,615 & 0 & 1,615 & 1,615 & 0 \\
        & 250 & \textbf{3,325} & 5,814 & $*$ & \textbf{3,325} & 4,495 & $*$ & \textbf{3,325} & 4,913.3 & $*$ & \textbf{3,325} & 5,323.2 & $*$ \\ 
        \midrule
		\multirow{4}{*}{50} & 50 & \textbf{1,908} & 2,323 & $*$ & 1,481 & 2,389.2 & $*$ & 1,495 & 2,434.2 & $*$ & 1,448 & 2,449.0 & $*$ \\ 
		& 100 & 3,148 & 4,775 & $*$ & 3,136 & 4,443.0 & $*$ & 3,136 & 4,619.6 & $*$ & \textbf{3,316} & 4,750.7 & $*$ \\
        & 250 & \textbf{6,945} & 12,156 & $*$ &\textbf{ 6,945} & 10,597.0 & $*$ & \textbf{6,945} & 11,180.3 & $*$ & \textbf{6,945} & 12,033.9 & $*$ \\
        & 500 & \textbf{14,470} & 24,418 & $*$ & \textbf{14,470} & 23,714.3 & $*$ & \textbf{14,470} & 24,327.7 & $*$ & 14,284 & 24,608.3 & $*$ \\ 
        \midrule
		\multirow{4}{*}{75} & 75 & \textbf{4,560} & 5,350 & $*$ & 0 & 5,485.8 & $*$ & 3,660 & 5,521.9 & $*$ & 3,554 & 5,536.3 & $*$ \\ 
		& 150 & \textbf{6,642} & 11,041 & $*$ & 6,144 & 10,710.4 & $*$ & 5,950 & 10,992.3 & $*$ & 6,475 & 11,121.3 & $*$ \\
        & 375 & \textbf{17,091} & 27,694 & $*$ & \textbf{17,091} & 17,742.2 & $*$ & \textbf{17,091} & 17,134.9 & $*$ & \textbf{17,091} & 17,669.4 & $*$ \\
        & 750 & 30,629 & 56,115 & $*$ & \textbf{31,145} & 31,493 & $*$ & \textbf{31,145} & 32,564.1 & $*$ & 31,139 & 48,253.3 & $*$ \\ 
        \midrule
		\multirow{4}{*}{100} & 100 & \textbf{7,278} & 9,631 & $*$ & 0 & 9,817.2 & $*$ & 0 & 9,866.7 & $*$ & 0 & 9,887.7 & $*$ \\ 
		& 200 & \textbf{12,297} & 19,732 & $*$ & 11,290 & 19,414.6 & $*$ & 11,656 & 19,620.1 & $*$ & 11,444 & 19,805.3 & $*$ \\
        & 500 & \textbf{28,160} & 49,507 & $*$ & 27,187 & 45,742.9 & $*$ & 27,187 & 46,467.7 & $*$ & 27,181 & 48,517.1 & $*$ \\
        & 1000 & 53,118 & 99,935 & $*$ & 55,240 & 94,771.8 & $*$ & \textbf{55,290} & 97,563.4 & $*$ & \textbf{55,290} & 97,887.1 & $*$ \\ 
		\bottomrule 
	\end{tabular} 
  	\end{adjustbox}
\end{table}

The results in Table~\ref{table:results} show that the problem-specific ILP of \citep{biro2021complexity} is highly competitive, obtaining the best solutions on most instances and proving optimality on the five smallest ones. Because of the heuristic choice of $\bm{\bar{w}}$, the true optimum can lie outside the feasible region of the bilinear models. For example, with $10$ agents and $50$ items the ILP proves optimality of a solution with value 223, while all bilinear models terminate below this value. On instances with many agents, the bilinear formulations sometimes fail to return primal solutions within the time limit. Still, our generic formulations outperform the ILP on several instances with many items, such as the one with 75 agents and 750 items. As expected, the invalid upper bounds of the bilinear models grow with $\bm{\bar{w}}$, but this does not yield a consistent improvement in primal solutions. Finally, the results confirm the difficulty of maximizing welfare among Pareto-optimal allocations, as no method terminates within the time limit once the number of agents reaches $50$.

\bibliographystyle{abbrvnat}
\bibliography{references}

\end{document}